\documentclass[3p]{elsarticle}

\journal{arXiv}

\usepackage{amsfonts}
\usepackage{amsmath, amscd, amssymb, bm, amsbsy, epsf, amsthm}
\usepackage{enumerate}
\usepackage{paralist}
\usepackage{graphicx,color}
\usepackage{subfigure}
\usepackage{mathrsfs}
\usepackage{verbatim}
\usepackage{inputenc}
\usepackage{diagbox}
\usepackage{algorithm}
\usepackage{algorithmicx}
\usepackage{algpseudocode}
\usepackage{hyperref}
\usepackage{pifont}
\usepackage{inputenc}
\usepackage{bbm, dsfont}
\usepackage{multicol}
\usepackage{balance}
\usepackage{verbatim, booktabs}

\usepackage{multirow, url}
\usepackage[table]{xcolor}
\usepackage{diagbox}
\usepackage{tikz}
\usepackage{capt-of}
\usepackage{natbib}

\usepackage{cmbright}

\DeclareMathAlphabet{\mathpzc}{OT1}{pzc}{m}{it}
\DeclareMathAlphabet{\mathpzc}{OT1}{pzc}{m}{it}

\newtheorem{theorem}{Theorem}

\newtheorem{proposition}[theorem]{Proposition}
\newdefinition{definition}{Definition}
\newdefinition{hypothesis}{Hypothesis}
\newdefinition{problem}{Problem}
\newdefinition{remark}{Remark}
\newdefinition{example}{Example}

\def\N{\boldsymbol{\mathbbm{N}}}
\def\R{\boldsymbol{\mathbbm{R}}}
\def\defining{\overset{\mathbf{def}}=}

\def\S{\mathcal{S}}

\def\Exp{\boldsymbol{\mathbbm{E}}}
\def\prob{\boldsymbol{\mathbbm{P} } }

\def\x{\mathbf{x} }

\DeclareMathAlphabet{\mathpzc}{OT1}{pzc}{m}{it}
\newcommand\Z{\mathop{}\!\mathbf{Z} } 

\newcommand\gr{\mathop{}\!\mathrm{G} } 
\newcommand\lr{\mathop{}\!\mathrm{LR} } 

\newcommand\upC{\mathbf{C}}

\newcommand\upP{\mathbf{P}}
\newcommand\upW{\mathbf{W}}

\newcommand\lt{\mathop{}\!\mathrm{left} } 
\newcommand\rt{\mathop{}\!\mathrm{right} } 
\newcommand\alg{\mathop{}\!\mathrm{alg} } 
\newcommand\side{\mathop{}\!\mathrm{side} } 
\newcommand\dc{\mathop{}\!\mathrm{DC} } 
\newcommand\mt{\mathop{}\!\mathrm{mt} } 

\newcommand\nfd{\mathop{}\!\mathrm{NFD} } 
\newcommand\ffd{\mathop{}\!\mathrm{FFD} } 
\newcommand\bfd{\mathop{}\!\mathrm{BFD} } 

\newcommand\ms{\mathop{}\!\mathrm{MS} } 
\newcommand\ma{\mathop{}\!\mathrm{MA} } 
\newcommand\nms{\mathop{}\!\mathrm{NMS} } 
\newcommand\case{\mathop{}\!\mathrm{case}}
\newcommand\Case{\mathop{}\!\mathrm{case}}

\begin{document}

\begin{frontmatter}

\title{Application and Assessment of Divide-and-Conquer-based \\
	Heuristic Algorithms 
	for some Integer Optimization Problems}
\tnotetext[mytitlenote]{This material is based upon work supported by project HERMES 54748 from Universidad Nacional de Colombia,
Sede Medell\'in.}

\author[mymainaddress]{Fernando A Morales} 
\cortext[mycorrespondingauthor]{Corresponding Author}
\ead{famoralesj@unal.edu.co}




\address[mymainaddress]{Escuela de Matem\'aticas
Universidad Nacional de Colombia, Sede Medell\'in \\
Carrera 65 \# 59A--110, Bloque 43, of 106,
Medell\'in - Colombia\corref{mycorrespondingauthor}}



\begin{abstract}
In this paper three heuristic algorithms using the Divide-and-Conquer paradigm are developed and assessed for three integer optimizations problems: Multidimensional Knapsack Problem (d-KP), Bin Packing Problem (BPP) and Travelling Salesman Problem (TSP). For each case, the algorithm is introduced, together with the design of numerical experiments, in order to empirically establish its performance from both points of view: its computational time and its numerical accuracy. 
\end{abstract}

\begin{keyword}
Divide-and-Conquer Method, Multidimensional Knapsack Problem, Bin Packing Problem, Traveling Salesman Problem, Monte Carlo simulations, method's efficiency.
\MSC[2010] 90C10 \sep 68Q87 
\end{keyword}

\end{frontmatter}



%
%
%
%
%
%
%
%
%
\section{Introduction}   
%
%
%
%
Broadly speaking, the Divide-and-Conquer method for Integer Optimization problems aims to reduce the computational complexity of the problem, at the price of loosing accuracy on the solutions. So far, it has been successfully introduced in \cite{MoralesMartinez} by Morales and Mart\'inez for solving the 0-1 Knapsack Problem (0-1KP). Its success lies on the fact that it is a good \textbf{balance} between computational complexity and solution quality.  It must be noted that the method does not \textbf{compete} with existing algorithms, it \textbf{complements} them. In addition, the method is ideally suited for parallel implementation, which is crucial in order to attain its full advantage. The aim of this paper is to extend the Divide-and-Conquer method and asses its performance to three Integer Optimization problems: Multidimensional Knapsack Problem (d-KP), Bin Packing Problem (BPP) and the Traveling Salesman Problem (TSP).

Throughout the problems presented here (as well as 0-1KP in \cite{MoralesMartinez}), the Divide-and-Conquer method consists in dividing an original integer optimization problem $ \Pi $, in two subproblems, namely $ \Pi_{\lt} $, $ \Pi_{\rt} $, using a \textbf{greedy function} as the main criterion for such subdivision. The process can be recursively iterated several times on the newly generated subproblems, creating a tree whose nodes are subproblems and whose leaves have an adequate (previously decided) size. The latter are the subproblems to be solved, either exactly or approximately by an \textbf{oracle algorithm} (an existing algorithm) chosen by the user according to the needs.  Finally, a \textbf{reassemble process} is done on the aforementioned tree to recover a feasible solution for the original problem $ \Pi $; this is the solution delivered by the method.

The 0-1KP has been extensively studied from many points of view, in contrast d-KP has been given little attention despite its importance. Among the heuristic methods for approximating d-KP the papers \cite{Loulou1979NewGH, Frville1986HeuristicsAR, diubin2008greedy} work on greedy methods, see \cite{frieze1984approximation, oguz1980polynomial} for heuristics based on the dual simplex algorithm (both pursuing polynomial time approximations), see \cite{Gavish1985EfficientAF} for a constraint-relaxation based strategy, see \cite{Volgenant1990AnIH} for a Lagrange multipliers approach and \cite{Chu1998AGA} for a genetic algorithm approach. Finally, the generalized multidemand multidimensional knapsack problem is treated heuristically in \cite{LaiTwoStage} using tabu search principles to detect cutting hyperplanes. As important as all these methods are, it can be seen that none of them introduces the Divide-and-Conquer principle; matter of fact, this paradigm was introduced to the 0-1KP for the first time in \cite{MoralesMartinez}. It follows that the strategy introduced in Section \ref{Sec The Problem d-KP} is not only new (to the Author's best knowledge) but it is complementary to any of the aforementioned methods.

The BPP ranks among the most studied problems in combinatorial optimization. Most of the efforts are focused on describing its approximation properties (see \cite{FernandezdelaVega1981BinPC, Karmarkar1982}) or finding lower bounds for the worst case performance ratio of the proposed heuristic algorithms (see \cite{YaoBPP, VanVlietBPP}). It is important to stress that among the many heuristics introduced to solve BPP, the algorithm presented in \cite{LeeOnlineBPP} follows the Divide-and-Conquer strategy, although it uses a harmonic partition of the items; which is different from the partition presented in Section \ref{Sec Bin Packing Problem BPP}. Yet again, the analysis presented by the Authors concentrates on lower bounds determination for the worst case scenario. Furthermore, from the worst case performance point of view, the classic Best Fit and First Fit algorithms have the same behavior as a large class of packing rules. It follows that a expected performance ratio analysis is needed in order to have more insight on the algorithm's performance; however, little attention has been given in this direction (see \cite{CoffmanBPPSurvey}), which is the target of our analysis here. As before the Divide-and-Conquer method combines perfectly with any of the preexisting ones, although the analysis here will be limited to its interaction with the most basic algorithms.

Finally, from the perspective of the TSP, the Divide-and-Conquer method is included in the family of heuristics reducing the size of the problem using some criteria, fixing edges \cite{FischeReducingTSP}, introducing multiple levels of reduction \cite{Walshaw2002AMA, Walshaw2004MultilevelRF} or the introduction of sub-tours from a previously known one \cite{TSPKazuonoriDC}. The sub-problems are solved using a known algorithm (or oracle) whether exact \cite{Applegate1998, ApplegateBCC03, Dantzig1954} or approximate \cite{Lin1973AnEH, LocalTSPJohnson1990}. Next, a merging tour algorithm \cite{CookMerging2003, TSPApplegatChained} is used to build a global tour. The present work applied to TSP in Section \ref{Sec Traveling Salesman Problem TSP} has the structure just described. 

The paper is organized as follows. In the rest of this section the notation is introduced, as well as the common criteria for designing the numerical experiments. Each of the analyzed problems has a separate section, hence Section \ref{Sec The Problem d-KP} is dedicated for d-KP, Section \ref{Sec Bin Packing Problem BPP} for BPP and Section \ref{Sec Traveling Salesman Problem TSP} is devoted to TSP. In the three cases, the section starts with a brief review of the problem, continues applying the Divide-and-Conquer principle for the problem at hand and closes with the corresponding numerical experiments (description and results). Finally, Section \ref{Sec Conclusions and Final Discussion} delivers the conclusions and closing discussion of the work. 
%
%
%
%
%
%
\subsection{Notation}\label{Sec Notation}
%
%
%
%
%
%
In this section the mathematical notation is introduced. For any natural number $ N \in \N $, the symbol $ [N] \defining \{ 1, 2, \ldots , N \} $ indicates the sorted set of the first $ N $ natural numbers. 
A particularly important set is $ \mathcal{S}_{N} $, where $ \mathcal{S}_{N} $ denotes the collection of all permutations in $ [N] $. 
For any real number $ x \in \R $ the floor and ceiling functions are given (and denoted) by $ \lfloor x \rfloor \defining \max\{k: k \leq x, \, k \text{ integer}\} $, $ \lceil x \rceil \defining \max\{k: k\geq x, \, k \text{ integer}\} $, respectively. Given an instance of a problem the symbols $ z_{\dc} $ and $ z^{*} $ indicate the optimal solution given by the presented Divide-and-Conquer algorithm and the optimal solution respectively. The analogous notation is used for $ T_{\dc}, T^{*} $ for the corresponding computational times.
%
%
\subsection{Numerical Experiments Design}\label{Sec Numerical Experiments Design}
%
%
The numerical experiments are aimed to asses the performance of the Divide-and-Conquer algorithms presented here. Our ultimate goal is to reliably compute the expected solution's accuracy and the computational time furnished by the method. To that end the probabilistic approach is adopted, whose construction is based on the Law of Large Numbers and the Central Limit Theorem  (which is written below for the sake of completeness). The theorem \ref{Th the Law of Large Numbers} assures the convergence of the random variables while the theorem \ref{Them Confidence Interval} yields the number of necessary trials, in order to assure that our computed (averaged) values, lie within a confidence interval centered at the actual expected value.
\begin{theorem}[Law of Large Numbers]\label{Th the Law of Large Numbers}
	Let $ \big(\Z^{(n)}:n\in \N\big) $ be a sequence of independent, identically distributed random variables with expectation $ \Exp\big(\Z^{(1)}\big) $, then
	\begin{equation}\label{Eq the Law of Large Numbers}
	\prob\bigg[ \Big\vert \frac{\Z^{(1)} + \Z^{(2)} + \ldots + \Z^{(n)}}{n} - 
	\Exp\big(\Z^{(1)}\big) \Big\vert > 0 \bigg]
	\xrightarrow[n\, \rightarrow \,\infty]{} 0 ,
	\end{equation}
	i.e., the sequence $ \big(\Z^{(n)}:n\in \N\big) $  converges in probability to its expectation $ \Exp(\Z^{(1)}) $, in the Ces\`aro sense.
\end{theorem}
\begin{proof}
	The proof and details can be found in \cite{BillingsleyProb}.
\end{proof}
%
%
%
\begin{theorem}\label{Them Confidence Interval}
	Let $ x $ be a scalar statistical variable with mean $ \bar{x} $ and variance $ \sigma^{2} $.
	\begin{enumerate}[(i)]
		\item The \textbf{number of Bernoulli trials} necessary to get a 95\% confidence interval is given by 
		\begin{equation}\label{Eq Bernoulli Trials for Confidence Intervals}
		n \defining \big( \frac{1.96}{0.05}\big)^{2} \sigma^{2}.
		\end{equation}

		\item The \textbf{95 percent confidence interval} is given by 
		\begin{equation}\label{Eq Confidence Interval Real-Valued}
		I_{x} \defining \Bigg[ \bar{x} - 1.96 \, \sqrt{ \frac{\sigma^{2}}{n} }, \bar{x} + 1.96 \,\sqrt{ \frac{\sigma^{2}}{n} } \, \Bigg] .
		\end{equation}
	\end{enumerate}
\end{theorem}
\begin{proof}
	The proof is based on the Central Limit Theorem, see \cite{Thompson} for details.
\end{proof}
%

\section{The Mutidimensional Knapsack Problem (d-KP)}\label{Sec The Problem d-KP}
%
%
In the current section the Multidimensional Knapsack Problem (d-KP) is recalled and a Divide-and-Conquer based algorithm, devised for this problem is presented. The technique uses a strategic choice of efficiency coefficients. 
%
%
\begin{problem}[Multidimensional Knapsack Problem, d-KP]\label{Pblm Original d-KP}
	Consider the problem
	\begin{subequations}\label{Eqn Original MKP Problem}
		\begin{equation}\label{Eqn MKP Objective Function}
		z^{*} \defining \max \sum\limits_{j\, = \, 1 }^{ N } p(j) \, x(j) ,
		\end{equation}
		subject to
		\begin{align}\label{Eqn MKP Capacity Constraints}
		& \sum\limits_{j \, = \, 1 }^{ N } w(i, j) \, x(j) \leq c(i), &
		& \text{for all } i \in [D] .
		\end{align}
		\begin{align}\label{Eqn MKP Choice Constraint}
		& x(j) \in \{0,1\}, &
		& \text{for all } j \,\in \,[N] .
		\end{align}	
	\end{subequations}
	Here, $\big( c(i) \big)_{i = 1}^{D} $ is the list of \textbf{knapsack capacities}, $ \big\{ \big(w(i, j)\big)_{j = 1}^{N}, i \in [D] \big\} $ is the list of corresponding weights and $  \big( x(j) \big)_{j = 1}^{N} $ is the list of binary-valued \textbf{decision variables}. In addition, the \textbf{weight coefficients} $ \big\{ \big(w(i, j)\big)_{j = 1}^{N}, i \in [D] \big\} $, as well as the knapsack capacities $ \big( c(i) \big)_{i = 1}^{N} $
	are all positive integers. 
	In the following
	$ z^{*} , \x^{*} $ denote the \textbf{optimal} objective function value and solution respectively. 
	The parameters $ \big(p(j) \big)_{j = 1}^{N} \subseteq (0, \infty]^{N} $ are known as the \textbf{profits}. Finally, in the sequel one of the d-KP instances will be denoted by $ \Pi = \big\langle (c(i))_{i \in [D]}, (p(j))_{j \in [N]}, \big\{ w(i, j) : j \in [N], i \in [D] \big\} \big\rangle $.	
\end{problem}
The following efficiency coefficients are introduced, they are scaled according to the respective constraint capacities
\begin{definition}[d-KP efficiency coefficients]\label{Def Efficiency Coefficient}
	For any instance of d-KP (Problem \ref{Pblm Original d-KP}), the \textbf{efficiency coefficients} are defined by
	\begin{align}\label{Eq Coefficient}
	& g(j) \defining \frac{p(j)}{\sum\limits_{i \,= \,1}^{D} \dfrac{w(i,j)}{c(i)} }, 
	& & j \in [N] .
	\end{align}
\end{definition}
Before continuing the analysis, the next hypothesis is adopted.
\begin{hypothesis}\label{Hyp Non Triviality and Sorting}
	In the sequel it is assumed that the instances $ \Pi $ of the d-KP satisfy the following conditions
	%
	\begin{subequations}\label{Subeq Weights non-triviality conditions}
		\begin{align}\label{Ineq All items are eligible}
		& w(i, j) \leq c(i), & & \text{ for all } i \in [N], j \in [D], 
		\end{align}
		\begin{align}\label{Ineq All constraints are non-trivial}
		& \sum\limits_{j \, = \, 1}^{N} w(i, j) > c(i) , &
		& \text{for all } i \in [D].
		\end{align}
	\end{subequations}
	%
\end{hypothesis}
\begin{remark}[d-KP Setting]\label{Rem Setting of 0-1 KP}
	
	%
	The condition \eqref{Subeq Weights non-triviality conditions} in \textsc{Hypothesis} \ref{Hyp Non Triviality and Sorting} guarantees two things. First, every item is eligible to be chosen (Inequality \eqref{Ineq All items are eligible}). Second, none of the constraints can be dropped (Inequality \eqref{Ineq All constraints are non-trivial}). 
	
	%
\end{remark}
\subsection{A Divide-and-Conquer Approach}
%
%
%
%
A Divide-and-Conquer method for solving the 0-1KP was introduced in \cite{MoralesMartinez}. There, an extensive discussion (both, theoretical and empirical) was presented on the possible strategies to implement it. Here, its conclusions are merely adjusted to yield the following method
\begin{definition}[A Divide-and-Conquer Algorithm for d-KP]
	\label{Def Divide and Conquer for d-KP}
	Let  \newline
	$ \Pi = \big\langle (c(i))_{i \in [D]}, (p(j))_{j \in [N]}, \big\{ w(i, j) : j \in [N], i \in [D] \big\} \big\rangle  $ be an instance of \textsc{Problem} \ref{Pblm Original d-KP}
	\begin{enumerate}[(i)]
		\item Sort the items in decreasing order according to the efficiency coefficients and re-index them so that
		\begin{equation}\label{Ineq Sorting}
		g(1) \geq 
		g(2) \geq 
		\ldots \geq 
		g(N) .
		\end{equation}
		
		\item Define $ V_{\lt} =  \{ j \in [N]: j \text{ is odd}\} $, $ V_{\rt} =  \{ j \in [N]: j \text{ is even}\}$ and
		\begin{align*}
		& c_{\lt}(i) =  \bigg\lceil \frac{\sum_{j \in V_{\lt}} w(i, j)}{\sum_{j = 1}^{n} w(i, j)} \bigg\rceil , &
		& c_{\rt}(i) = c(i) - c_{\lt}(i) , & 
		& \text{for all } i \in [D].
		\end{align*}

		\item A Divide-and-Conquer pair of \textsc{Problem} \ref{Pblm Original d-KP} is the couple of subproblems $ \big(\Pi_{\side}: \, \side \in \{\lt, \rt\} \big) $, each with input data \newline
		$ \Pi_{\side} = 
		\big\langle (c_{\side}(i))_{i \in [D]}, \big( p(j)\big)_{ j \, \in \, V_{\side} }, \big( w(i, j)\big)_{ j \, \in \, V_{\side} } \big\rangle $.
		In the sequel, \newline
		$ \big(\Pi_{\side}, s = \lt, \rt \big) $ is referred as a \textbf{D\&C pair} $ z_{\side}^{*} $ denotes the optimal solution value of the problem $ \Pi_{\side} $.

		\item The D\&C solution is given by
		\begin{align}\label{Eq DAC tree solution}
		& \x^{\mt}_{\dc} \defining \x^{\mt}_{\lt} \cup \x^{\mt}_{\rt}, &
		& z_{\dc}^{\mt} \defining z^{\mt}_{\lt} + z^{\mt}_{\rt}  .
		\end{align}
		Here, the index $ \mt $ indicates the method used to solve the problem, in this work and particular problem, only $ \mt =* $ is used.
		Also, some abuse of notation is introduced, denoting by $ \x^{\mt}_{\side} $ the solution of $ \Pi_{\side} $ and using the same symbol as a set of chosen items (instead of a vector) in the union operator. In particular, the maximal possible value occurs when all the summands are at its max i.e., the method approximates the optimal solution by the feasible solution $ \x^{\mt}_{\dc} = \x^{\mt}_{\lt} \cup \x^{\mt}_{\rt}  $ with objective value $ z^{\mt}_{\dc} = z^{\mt}_{\lt} + z^{\mt}_{\rt} $.
	\end{enumerate}
\end{definition}
\begin{example}[Divide-and-Conquer Algorithm on d-KP]\label{Exm d-KP Instance}
	Consider the d-KP instance described by the table \ref{Tbl Example Problem}, with knapsack capacities $ c(1) = 16, c(2) = 11 $ and number of items $ N = 6 $. 
	\begin{table}[h!]
		\begin{centering}
			\rowcolors{2}{gray!25}{white}
			\begin{tabular}{c | c c c c c c }
				\rowcolor{gray!80}
				\hline
				$ j $ & 1 & 2 & 3 & 4 & 5 & 6 
				\\
				$ p(j) $ & 5 & 11 & 11 &71 & 2 & 2  \\
				$ w(1, j) $ & 6 & 7 & 1 & 7 & 7 & 4 \\
				$ w(2,j) $ & 4 & 1 & 1 & 6 & 1 & 8 \\
				$ g(j) $ & 13.53 & 41.01 & 144.03 & 14.28 & 7.46 & 4.14 \\
				\hline		
			\end{tabular}
			\caption{d-KP problem of \textsc{Example} \ref{Exm d-KP Instance}, knapsack capacities $ c(1) = 16, c(2) = 11 $, number of items $ N = 6 $, $ D = 2 $.}
			\label{Tbl Example Problem}
		\end{centering}
	\end{table}

	\noindent In this particular case the D\&C pair is given by
	\begin{align*}
	& \Pi_{ \lt }: &
	V_{ \lt} & = [3, 4, 5 ], &
	& c_{\lt}(1) =  8,  & 
	& c_{\lt}(2) = 5,
	\\
	& \Pi_{ \rt }: &
	V_{ \rt} & = [2, 1, 6 ], &
	& c_{\rt}(1) =  8,  & 
	& c_{\rt}(2) = 6.
	\end{align*}
	%
	%
\end{example}
\begin{proposition}\label{Thm quality certificate DCM}
	Let  $ \Pi = \big\langle (c(i))_{i \in [D]}, (p(j))_{j \in [N]}, \big\{ w(i, j) : j \in [N], i \in [D] \big\} \big\rangle  $ be an instance of d-KP (problem \ref{Pblm Original d-KP}). Let $ (\Pi_{\lt}, \Pi_{\rt} ) $ be the associated Divide-and-Conquer pair introduced in Definition \ref{Def Divide and Conquer for d-KP}.
	Then, 
	\begin{subequations}\label{Subeq quality certificate DCM}
		\begin{align}\label{Eq feasibility DCM}
		& \x_{\dc}^{\mt} = \x_{\lt}^{\mt} \cup \x_{\rt}^{\mt} \text{ is } \Pi\text{-feasible} , &
		& \text{for } \mt = \gr, \, \ast, \, \lr ,
		\end{align}
		\begin{align}\label{Ineq quality certificate DCM}
		& z^{\mt}_{\lt} + z^{\mt}_{\rt} = z_{\dc}^{\mt} \leq 
		z^{\mt}, & & \text{for } \mt = \ast, \, \lr .
		\end{align}
	\end{subequations}
	Here, $ \gr $ and $ \lr $ indicate any greedy method and $ \lr $ the linear relaxation.
\end{proposition}
\begin{proof}
	Trivial.
\end{proof}
As already stated, the purpose of this work is to determine experimentally, under which conditions an instance $ \Pi $ of d-KP can be solved using the proposed method in an justifiable way. That is, if the trade between computational complexity and accuracy is acceptable. To this end, a set of parameters is introduced, which has proved to be useful in order to classify the d-KP instances, within these classifications the output of numerical methods is comparable, see \cite{CORAL}. 
\begin{definition}[Tightness Ratios]\label{Def Tightness}
	Let $ \Pi = \big\langle (c(i))_{i \in [D]}, (p(j))_{j \in [N]}, \big\{ w(i, j) : j \in [N], i \in [D] \big\} \big\rangle  $ be an instance of d-KP, then the set of tightness ratios is given by
	\begin{align}\label{Eq Tightness}
	& t(i) \defining \frac{c(i)}{\sum_{j = 1}^{N} w(i, j)} , &
	& \text{for all } i \in [D].
	\end{align}
	(Observe the differences with respect to the efficiency coefficients introduced in Definition \ref{Def Efficiency Coefficient}.)
\end{definition}
%
%
%
\subsection{Numerical Experiments}\label{Sec Numerical Experiments d-KP}
%
In order to numerically asses the method's effectiveness on d-KP, the numerical experiments are designed according to its main parameters i.e., size ($ N \in \{6, 10, 20, 50, 100, 250, 500, 1000\} $), restrictiveness ($ D  \in \{2, 3, 4, 5, 6\} $) and tightness ($  t = \{0.25, 0.5, 0.75 \} $). 
While the values of the first three parameters $ N, D, t  $ were decided based on experience (see  \cite{CORAL}), the number of trials $ k = 1000 $, was decided based on \textsc{Equation} \eqref{Eq Bernoulli Trials for Confidence Intervals}. The variance $ \sigma^{2} $ was chosen as the worst possible value for a sample of 100 trials, which delivered an approximate value $ k = 1000 $. (After, executing the 1000 experiments the variance showed consistency with the initial value computed from the 100 trial-sample.)

The random instances $ \Pi = \big\langle (\upC(i))_{i \in [D]}, (\upP(j))_{j \in [N]}, \big\{ \upW(i, j) : j \in [N], i \in [D] \big\} \big\rangle  $ were generated as follows. 
\begin{enumerate}[a.]
	\item The profits $ (\upP(i))_{i \in [N]} $ will be independent random variables with uniform distribution on the interval $ [1, N\cdot D] $ (i.e., $ \prob\big(\upP(i)  = \ell \big) = \frac{1}{N\cdot D} $ for all $ \ell \in [1, N\cdot D] $ and $ i \in [N] $). 
	
	\item Initial capacities $ (\upC^{(0)}(i))_{i \in [D]} $ are generated as independent random variables uniformly distributed on the interval $ [1, N\cdot D] $. 
	
	\item For all $ i \in [D] $, the weights $  (\upW(i,j))_{j \in [N]} $ are almost independent random variables, uniformly distributed in the interval $ [1, \upC^{(0)}(i)] $. Next, the capacities are computed using $ \upC_{i} = t(i)  \sum_{j = 1}^{N} \upW(i, j)  $, where $ t(i) $ is the tightness ratio introduced in Definition \ref{Def Tightness} and decided a-priori.
\end{enumerate}
\begin{definition}[d-KP Performance Coefficients]\label{Def Solution and Time Performance Coefficients}
	The d-KP performance coefficients are given by the percentage fractions of the numerical solution and the computational time given by the proposed algorithm with respect to the exact solution when using the same method of resolution for both cases, i.e., 
	\begin{align}\label{Eq Solution and Time Performance Coefficients}
	& S_{f} = 100 \times \dfrac{z_{\dc}}{z^{*}} ,  &
	& T_{f} = 100 \times \dfrac{T_{\dc}}{T^{*}} .
	\end{align}
\end{definition}
The exact solution of the d-KP problem was computed using the public Python library MIP, i.e., in our case the MIP subroutine is playing the role of the oracle solving the subproblems.  
The tables \ref{Tbl d-KP, t = 0.25}, \ref{Tbl d-KP, t = 0.5}, \ref{Tbl d-KP, t = 0.75} below, summarize the expected values of the performance coefficients $ S_{f} $ (solution fraction) and $ T_{f} $ (time fraction) (see \textsc{Definition} \ref{Def Solution and Time Performance Coefficients}) for the tightness values $ t = 0.25, 0.5, 0.75 $ respectively and dimensions $ D = 2, 4, 6 $. The corresponding graphs are depicted in the figures  \ref{Fig DC d-KP, t = 0.25},  \ref{Fig DC d-KP, t = 0.5} and  \ref{Fig DC d-KP, t = 0.75}

\begin{table}[h!]
	\begin{minipage}{0.55\textwidth}
		\begin{centering}
			\rowcolors{2}{gray!25}{white}
			\begin{tabular}{c | cc  cc  cc}
				\rowcolor{gray!80}
				\hline
				Items & \multicolumn{2}{ c  }{D = 2} & \multicolumn{2}{ c }{D = 4}
				& \multicolumn{2}{ c }{D = 6} \\
				\rowcolor{gray!80} $ N $ 
				& $  S_{f} $ & $ T_{f} $ 
				& $  S_{f} $ & $ T_{f} $ 
				& $  S_{f} $ & $ T_{f} $ 
				\\
				6	& 51.85	& 63.25	& 8.39	& 46.87	& 33.18	& 51.28 \\
				10	& 76.39	& 54.63	& 60.88	& 54.42	& 70.75	& 48.34 \\
				20	& 88.98	& 54.13	& 82.84	& 51.43	& 87	  & 53.65 \\
				50	& 96.36	& 51.34	& 93.69	& 52.39	& 95.69	& 51.96 \\
				100	& 98.67	& 52.09	& 97.2	& 51.49	& 97.53	& 51.84 \\
				250	& 99.63	& 53.47	& 99.03	& 51.65	& 99.01	& 51.86 \\
				500	& 99.84	& 51.63	& 99.55	& 50.9	& 99.47	& 51.07 \\
				1000 & 99.94 & 51.41 & 99.81 & 50.23 & 99.77 & 50.52 \\
				\hline		
			\end{tabular}
			\caption{Performance table for the Divide-and-Conquer method, solving d-KP problem. Tightness ratio $ t = 0.25 $. Bernoulli trials $ k = 1000 $. The performance coefficients $ S_{f} $ and  $ T_{f}$ are introduced in \textsc{Equation}  \eqref{Eq Solution and Time Performance Coefficients}}
			\label{Tbl d-KP, t = 0.25}
		\end{centering}
	\end{minipage}
	\begin{minipage}{0.45\textwidth}
		\centering
		{\includegraphics[scale = 0.5]{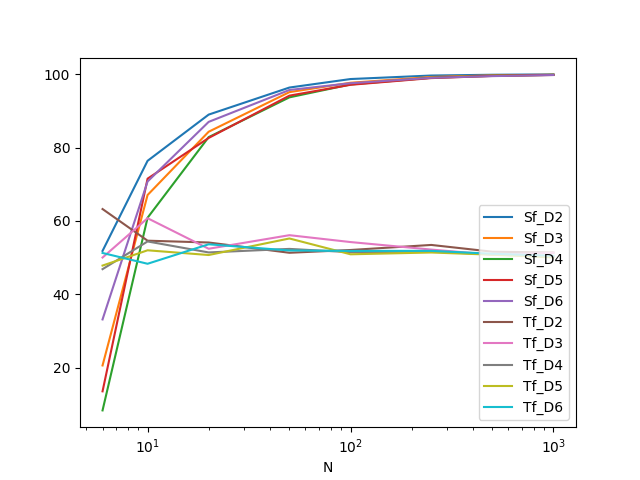}
		}
		\captionof{figure}{Graphic solution fraction $ S_{f} $ and time fraction $ T_{f} $ for the values displayed in table \ref{Tbl d-KP, t = 0.25}. The domain is given by the number of items $ N $ in logarithmic scale.}
		\label{Fig DC d-KP, t = 0.25}
	\end{minipage}
\end{table}
\begin{table}[h!]
	\begin{minipage}{0.55\textwidth}
		\begin{centering}
			\rowcolors{2}{gray!25}{white}
			\begin{tabular}{c | cc  cc  cc}
				\rowcolor{gray!80}
				\hline
				Items & \multicolumn{2}{ c  }{D = 2} & \multicolumn{2}{ c }{D = 4}
				& \multicolumn{2}{ c }{D = 6} \\
				\rowcolor{gray!80} $ N $ 
				& $  S_{f} $ & $ T_{f} $ 
				& $  S_{f} $ & $ T_{f} $ 
				& $  S_{f} $ & $ T_{f} $ 
				\\
				6	& 78.62	& 54.42	& 75.57	& 56.41	& 68.54	& 57.77 \\
				10	& 88.69	& 53.79	& 84.35	& 50.98	& 81.03	& 52.51\\
				20	& 94.49	& 52.57	& 92.56	& 49.24	& 90.45	& 53.64 \\
				50	& 98.35	& 52.44	& 96.96	& 51.44 & 	96.27 & 51.3 \\
				100 & 99.2	& 50.91	& 98.51	& 52.66	& 98.07	& 51.03 \\
				250	& 99.74	& 51.95	& 99.43	& 51.05	& 99.21	& 50.77 \\
				500	& 99.89	& 50.49	& 99.73	& 51.29	& 99.58	& 50.57 \\
				1000	& 99.94	& 50.98	& 99.89	& 51.13	& 99.8 & 50.72 \\
				\hline		
			\end{tabular}
			\caption{Performance table for the Divide-and-Conquer method, solving d-KP problem. Tightness ratio $ t = 0.5 $. Bernoulli trials $ k = 1000 $. The performance coefficients $ S_{f} $ and  $ T_{f}$ are introduced in \textsc{Equation}  \eqref{Eq Solution and Time Performance Coefficients}}
			\label{Tbl d-KP, t = 0.5}
		\end{centering}
	\end{minipage}
	\begin{minipage}{0.45\textwidth}
		\centering
		{\includegraphics[scale = 0.5]{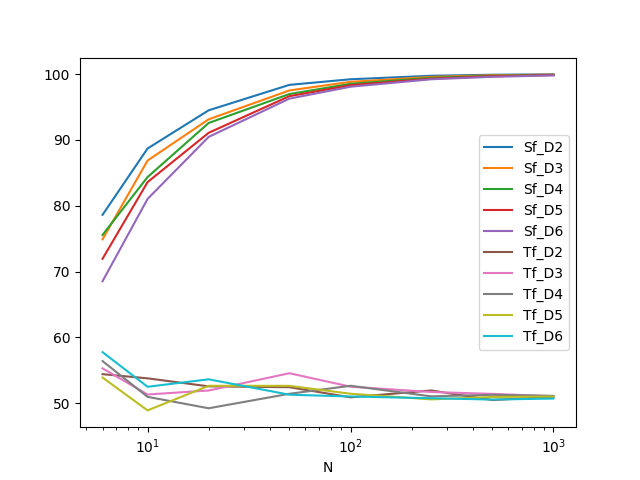}
		}
		\captionof{figure}{Graphic solution fraction $ S_{f} $ and time fraction $ T_{f} $ for the values displayed in table \ref{Tbl d-KP, t = 0.5}. The domain is given by the number of items $ N $ in logarithmic scale.}
		\label{Fig DC d-KP, t = 0.5}
	\end{minipage}
\end{table}
\begin{table}[h!]
	\begin{minipage}{0.55\textwidth}
		\begin{centering}
			\rowcolors{2}{gray!25}{white}
			\begin{tabular}{c | cc  cc  cc}
				\rowcolor{gray!80}
				\hline
				Items & \multicolumn{2}{ c  }{D = 2} & \multicolumn{2}{ c }{D = 4}
				& \multicolumn{2}{ c }{D = 6} \\
				\rowcolor{gray!80} $ N $ 
				& $  S_{f} $ & $ T_{f} $ 
				& $  S_{f} $ & $ T_{f} $ 
				& $  S_{f} $ & $ T_{f} $ 
				\\
				6 &	90.76 &	52.83 &	84.26 &	52.76 &	75.83 &	48.85 \\
				10 & 93.42 & 51.71 & 90.91 & 52.77 & 89.93 & 50.96 \\
				20	& 97.46	& 50.83	& 96.16 & 54.96	& 95.41	& 51.21 \\
				50 & 99.04 & 52.87 & 98.63 & 51.7 &	98.24 &	50.37 \\
				100	& 99.57	& 52.33 & 99.33 & 52.34	& 99.13	& 51.59 \\
				250	& 99.83 & 51.44	& 99.7	& 50.77 & 99.65	& 51.54 \\
				500 & 99.94	& 51.11	& 99.86 & 50.67	& 99.8	& 51.1 \\
				1000 & 99.97 & 51.14 & 99.94 & 51.16 & 99.9	& 50.38 \\
				\hline		
			\end{tabular}
			\caption{Performance table for the Divide-and-Conquer method, solving d-KP problem. Tightness ratio $ t = 0.75 $. Bernoulli trials $ k = 1000 $. The performance coefficients $ S_{f} $ and  $ T_{f}$ are introduced in \textsc{Equation}  \eqref{Eq Solution and Time Performance Coefficients}}
			\label{Tbl d-KP, t = 0.75}
		\end{centering}
	\end{minipage}
	\begin{minipage}{0.45\textwidth}
		\centering
		{\includegraphics[scale = 0.5]{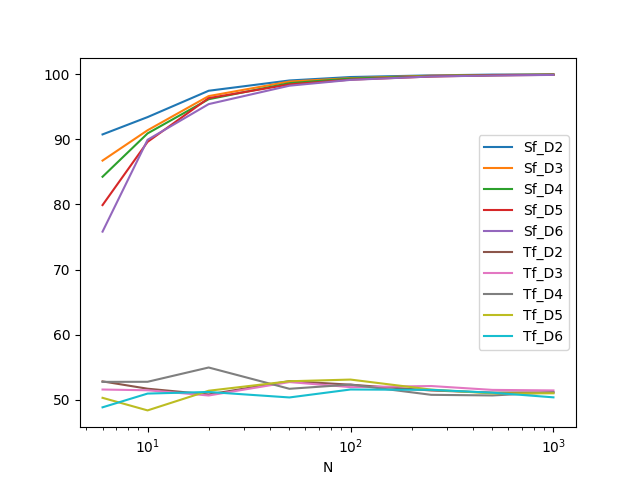}
		}
		\captionof{figure}{Graphic solution fraction $ S_{f} $ and time fraction $ T_{f} $ for the values displayed in table \ref{Tbl d-KP, t = 0.75}. The domain is given by the number of items $ N $ in logarithmic scale.}
		\label{Fig DC d-KP, t = 0.75}
	\end{minipage}
\end{table}
%
%
\section{The Bin Packing Problem BPP}\label{Sec Bin Packing Problem BPP}
%
%
In the current section the Bin Packing Problem (BPP) is introduced and a list of greedy algorithms is reviewed; these will be used in the following to compute/measure the expected performance of a Divide-and-Conquer based proposed method acting on BPP.

The Bin-Packing Problem is described as follows. Given $ N $ objects, each of a given weight $ (w(j))_{j = 1}^{N} \subseteq (0,c] $ and bins of capacity $ c $ (at least $ N $ of them), the aim is to assign the objects in a way that the minimum number of bins is used and that every object is packed in one of the bins.  Without loss of generality, the capacity of the items can be normalized to $ c = 1 $ and the weight of the items accordingly. Hence, the BPP formulates in the following way
\begin{problem}[Bin Packing Problem, BPP]\label{Pblm Original BPP}
	Consider the problem
	\begin{subequations}\label{Eqn Original BPP Problem}
		\begin{equation}\label{Eqn BPP Objective Function}
		z^{*} \defining \min \sum\limits_{i\, = \, 1 }^{ N } y(i) ,
		\end{equation}
		subject to
		\begin{align}\label{Eqn BPP Capacity Constraints}
		& \sum\limits_{j \, = \, 1 }^{ N } w(j) \, x(i, j) \leq y(j), &
		& \text{for all } i \in [N] .
		\end{align}
		\begin{align}\label{Eqn BPP Choice Constraint}
		& x(i, j), y(i) \in \{0,1\}, &
		& \text{for all }i,  j \,\in \,[N] .
		\end{align}	
	\end{subequations}
	Here, $\big( w(j) \big)_{j = 1}^{N} \subseteq (0,1] $ is the list of \textbf{item weights}, $ \big\{ x(i, j): i, j \in [N] \big\} $ is the list of binary valued \textbf{item assignment decision variables}, with $ x(i,j) = 1 $ if item $ j $ is assigned to bin $ i $ and $ x(i, j) = 0 $ otherwise. Finally, $  \big( y(i) \big)_{i = 1}^{N} $ is the list of binary valued \textbf{bin-choice decision variables}, setting $ y(i) = 1 $ if bin $ i $ is nonempty and $ y(i) = 0 $ otherwise.
\end{problem}
As it is well-know, most of the literature is concerned with heuristic methods for solving BPP (see \cite{CoffmanBPPSurvey}). Therefore, this work is focused on the interaction of a Divide-and-Conquer paradigm with the three most basic heuristic algorithms: Next Fit Decreasing (NFD), First Fit Decreasing (FFD) and Best Fit Decreasing (BFD) (see \cite{MartelloKnapsackPs} and \cite{KorteVygen}). For the sake of completeness these algorithms are described below. In the three cases, it is assumed that the items are sorted decreasingly according their weights, i.e. $ w(1) \geq w(2) \geq \ldots \geq w(N) $.
\begin{enumerate}[a.]
	\item \textbf{Next Fit Algorithm.} The first item is assigned to bin $ 1 $. Each of the items $ 2, \ldots, N $ is handled as follows: the item $ j $ is assigned to the current bin if it fits; if not it is assigned to a new bin (next fit).
	
	\item \textbf{First Fit Algorithm.} Consider the items increasingly according to its index. For each item $ j $, assign it to the first initialized bin $ i $ (first fit) where it fits (if any), and if it does not fit in any initialized bin, assign it to a new one. 
	
	\item \textbf{Best Fit Algorithm.} Consider the items increasingly according to its index. For each item $ j $, assign it to the bin $ i $ where it fits (if any) and it is as full as possible; if it does not fit in any initialized bin, assign it to a new one. 
\end{enumerate}
\subsection{The Divide-and-Conquer Approach}
%
%
%
%
The NF, FF and BF algorithms work exactly as NFD, FFD and BFD, except that they do not act on a sorted list of items. It has been established that the NF algorithm (and consequently FF as well as BF) has good performance on items of small weight (see \cite{KorteVygen}). Therefore, it is only natural to deal with the pieces according to their weight in decreasing order i.e., in the BPP problem the weights are used as greedy function (or efficiency coefficients) in order to propose the corresponding Divide-and-Conquer heuristic
\begin{definition}[Divide-and-Conquer Method for BPP]\label{Def Divide and Conquer for BPP}
	Let  \newline
	$ \Pi = \big\langle (w(j))_{j \in [N]} \big\rangle  $ be an instance of \textsc{Problem} \ref{Pblm Original BPP}
	\begin{enumerate}[(i)]
		\item Sort the items in decreasing order according to their weights and re-index them so that
		\begin{equation}\label{Ineq Sorting BPP}
		w(1) \geq 
		w(2) \geq 
		\ldots \geq 
		w(N) .
		\end{equation}
		
		\item Define $ V_{\lt} =  \{ j \in [N]: j \text{ is odd}\} $ and $ V_{\rt} =  \{ j \in [N]: j \text{ is even}\} $.  
		
		\item A Divide-and-Conquer pair of \textsc{Problem} \ref{Pblm Original BPP} is the couple of subproblems $ \big(\Pi_{\side}: \, \side \in \{\lt, \rt\} \big) $, each with input data $ \Pi_{\side} = 
		\big\langle ( w(j))_{ j \, \in \, V_{\side} } \big\rangle $.
		In the sequel, $ \big(\Pi_{\side}, s = \lt, \rt \big) $ is referred as a \textbf{D\&C pair}. 
		We denote by $ z_{\side}^{\alg} $ the solution value of the problem $ \Pi_{\side} $, while $ z^{\alg} $ denotes the solution value of the full problem $ \Pi $, furnished by the algorithm $ \alg $. 	

		\item The D\&C solution of \textsc{Problem} \ref{Pblm Original BPP} is given by
		\begin{align}\label{Eq DAC BPP tree solution}
		& \x^{\alg}_{\dc} \defining \x^{\alg}_{\lt} \cup \x^{\alg}_{\rt}, &
		& z_{\dc}^{\alg} \defining z^{\alg}_{\lt} + z^{\alg}_{\rt}  .
		\end{align}
		Here, the index $ \alg $ indicates the method used to solve the problem. This work, uses $ \alg = \nfd, \ffd, \bfd $, standing for the Next Fit Decreasing, First Fit Decreasing and Best Fit Decreasing algorithms respectively (described in the previous section). 
		Also, some abuse of notation is introduced, denoting by $ \x^{\alg}_{\side} = \{x^{\alg}_{\side}(i, j): i, j \in V_{\side} \} $ the optimal solution of $ \Pi_{\side} $ and using the same symbol as a set of chosen items (instead of a vector) in the union operator. In particular, the minimal possible value occurs when all the summands are at its minimum i.e., the method approximates the optimal solution by the feasible solution $ \x^{\alg}_{\dc} = \x^{\alg}_{\lt} \cup \x^{\alg}_{\rt}  $ with objective value $ z^{\alg}_{\dc} = z^{\alg}_{\lt} + z^{\alg}_{\rt} $.
	\end{enumerate}
\end{definition}
\begin{example}[Divide-and-Conquer Algorithm on BPP]\label{Exm BPP Instance}
	Consider the BPP instance described by the table \ref{Tbl BPP Example Problem}, number of items $ N = 6 $. 
	\begin{table}[h]
		\begin{centering}
			\rowcolors{2}{gray!25}{white}
			\begin{tabular}{c | c c c c c c }
				\rowcolor{gray!80}
				\hline
				$ j $ & 1 & 2 & 3 & 4 & 5 & 6 
				\\
				$ w(j) $ & 0.5 & 0.7 & 0.25 & 0.1 & 0.85 & 0.31  \\
				\hline		
			\end{tabular}
			\caption{BPP problem of \textsc{Example} \ref{Exm BPP Instance}, number of items $ N = 6 $.}
			\label{Tbl BPP Example Problem}
		\end{centering}
	\end{table}

	\noindent For this particular case the D\&C pair is given by
	\begin{align*}
	& \Pi_{ \lt }: &
	V_{ \lt} & = [ 5, 1, 3 ], 
	\\
	& \Pi_{ \rt }: &
	V_{ \rt} & = [2, 6, 4 ] .
	\end{align*}
\end{example}
\begin{proposition}\label{Thm feasibility DCM BPP}
	Let  $ \Pi = \big\langle (w(j))_{j \in [N]} \big\rangle  $ be an instance of the BPP problem \ref{Pblm Original BPP}. Let $ (\Pi_{\lt}, \Pi_{\rt} ) $ be the associated Divide-and-Conquer pair introduced in Definition \ref{Def Divide and Conquer for BPP}.
	Then, 
	\begin{align}\label{Eq feasibility DCM BPP}
	& \x_{\dc}^{\alg} = \x_{\lt}^{\alg} \cup \x_{\rt}^{\alg} \text{ is } \Pi\text{-feasible} , &
	& \text{for } \alg = \nfd, \,  \ffd, \, \bfd .
	\end{align}
\end{proposition}
\begin{proof}
	Trivial.
\end{proof}
\begin{remark}
	Observe that, unlike \textsc{Proposition} \ref{Thm quality certificate DCM}, here, there is no claim of an inequality of control such as \eqref{Ineq quality certificate DCM}. This is because the NFD, FFD and BFD methods are all heuristic. Hence, it can happen (and it actually does for some few instances) that $ z_{\dc}^{\alg} \geq z^{\alg} $ for $ \alg = \nfd, \ffd, \bfd $; it is worth noticing that this phenomenon is more frequent for the NFD method.
\end{remark}
%
%
\subsection{Numerical Experiments}\label{Sec Numerical Experiments BPP}
%
In order to numerically asses the effectiveness of the proposed Divide-and-Conquer algorithm, the experiments are designed according to its main parameter i.e., size ($ N \in \{20, 50, 100, 250, 500, 1000, 1500, 2000\} $) and number of trials ($ k $). The number of items was decided based on experience (more specifically the Divide-and-Conquer algorithm yields poor results for low values of $ N $ as it can be seen in the results), while the number of trials $ k = 1500 $ was chosen based on \textsc{Equation} \eqref{Eq Bernoulli Trials for Confidence Intervals}. The variance $ \sigma^{2} $ was chosen as the worst possible value for a sample of 100 trials, which delivered an approximate value $ k = 1500 $. (After, executing the 1500 experiments the variance showed consistency with the the initial value computed from the 100 trial-sample.)

For the random instances $ \Pi = \big\langle  (\upW(j))_{j \in [N]} \big\rangle  $, the weights $  (\upW(j))_{j \in [N]} $ are independent random variables uniformly distributed on the interval $ [0, 1] $ (notice that $ \prob(\upW(i) = 0 ) = 0 $ for all $ i \in [N] $). Next, the numerical and time performance coefficients are introduced.
\begin{definition}[BPP Performance Coefficients]\label{Def Solution and Time Performance Coefficients BPP}
	The BPP performance coefficients are given by the percentage fractions of the solution and the computational time given by the Divide-and-Conquer approach with respect to the exact solution, when using the same method of resolution for both cases, i.e., 
	\begin{align}\label{Eq Solution and Time Performance Coefficients BPP}
	& S_{f}^{\alg} = 100 \times \dfrac{z^{\alg}}{z^{\alg}_{\dc}} ,  &
	& T_{f}^{\alg} = 100 \times \dfrac{T_{\dc}^{\alg}}{T^{\alg}}, &
	& \alg = \nfd, \, \ffd, \, \bfd .
	\end{align}
	Here, $ z^{\alg}, z^{\alg}_{\dc} $ indicate the solution furnished for the full problem and for the Divide-and-Conquer heuristic respectively, when using the algorithm $ \alg $. 
\end{definition}
\begin{remark}\label{Ref Solution and Time Performance Coefficients BPP}
	The definition of $ S^{\alg}_{f} $ (equation \eqref{Eq Solution and Time Performance Coefficients BPP}) for BPP differs from that given in d-KP (equation \eqref{Eq Solution and Time Performance Coefficients}) in order to keep normalized the fraction. Since BPP is a minimization problem the Divide-and-Conquer solution will be larger than the exact solution ($ z^{\alg}_{\dc} \geq z^{\alg} $) for most of the cases. This contrasts with d-KP which is maximization problem, where the opposite behavior takes place ($ z_{\dc} \leq z^{*} $).
\end{remark}
The table \ref{Tbl BPP Global Table} below, summarizes the expected values of the performance coefficients $ S^{\alg}_{f} $ (solution fraction) $ T^{\alg}_{f} $ (time fraction) (see \textsc{Definition} \ref{Def Solution and Time Performance Coefficients BPP}). The corresponding graphs are depicted in the figures \ref{Fig BPP Global figures} (a) for $ S^{\alg}_{f} $ and (b) for $ T^{\alg}_{f} $, respectively.
\begin{table}[h]
	\begin{centering}
		\rowcolors{2}{gray!25}{white}
		\begin{tabular}{c | cc  cc  cc}
			\rowcolor{gray!80}
			\hline
			Items & \multicolumn{2}{ c  }{NFD} & \multicolumn{2}{ c }{FFD}
			& \multicolumn{2}{ c }{BFD} \\
			\rowcolor{gray!80} $ N $ 
			& $  S^{\nfd}_{f} $ & $ T^{\nfd}_{f} $ 
			& $  S^{\ffd}_{f} $ & $ T^{\ffd}_{f} $ 
			& $  S^{\bfd}_{f} $ & $ T^{\bfd}_{f} $ 
			\\
			20 & 99.03	& 126.86	& 96.57	& 70.56	& 96.57	& 92.7 \\
			50	& 99.77	& 112.98	& 98.55	& 59.75	& 98.59	& 82.73 \\
			100	& 99.83	& 111.01	& 98.92	& 54.66	& 98.92	& 76.72 \\
			250	& 99.99	& 102.36	& 99.59	& 52.09	& 99.58	& 65.09 \\
			500	& 99.99	& 99.97	& 99.84	& 50.26	& 99.84	& 58.35 \\
			1000 & 99.98 & 100.04 &	99.92 & 50.43	& 99.92	& 55.04 \\
			1500 & 99.99	& 101.05	& 99.94	& 50.83	& 99.94	& 53.85 \\
			2000 & 99.99	& 101.16 & 99.96 & 50.76 & 99.96 & 53.05 \\
			\hline		
		\end{tabular}
		\caption{Performance table for the Divide-and-Conquer method, when solving the BPP problem with the three methods: NFD, FFD and BFD. Bernoulli trials $ k = 1500 $. The performance coefficients $ S^{\alg}_{f} $ and  $ T^{a\lg}_{f} $ for $ \alg = \nfd, \ffd, \bfd $ are introduced in \textsc{Equation} \eqref{Eq Solution and Time Performance Coefficients BPP}.}
		\label{Tbl BPP Global Table}
	\end{centering}
\end{table}
\begin{figure}[h!]
	\begin{subfigure}[Solution Fraction $ S_{f} $ Performance.]
		{\includegraphics[scale = 0.5]{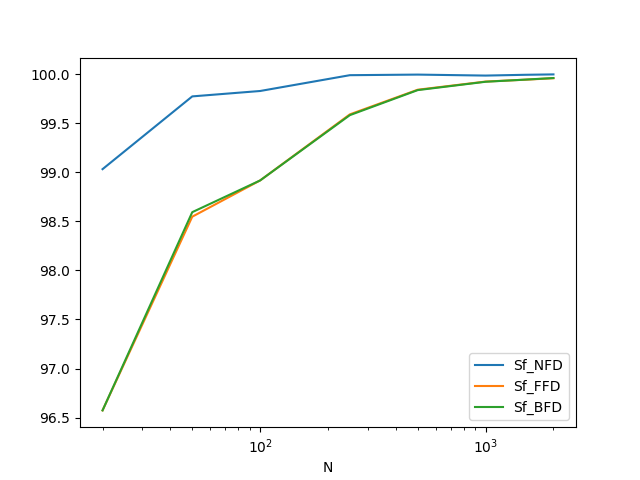}}
	\end{subfigure}
	\begin{subfigure}[Time Fraction $ T_{f} $ Performance.]
		{\includegraphics[scale = 0.5]{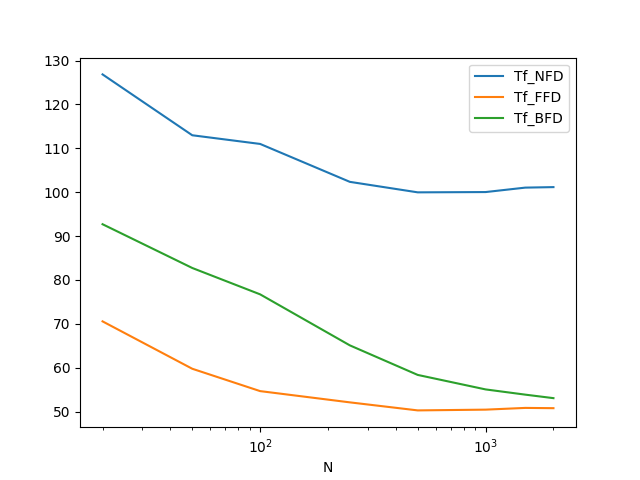}}
	\end{subfigure}
	\caption{The graphs above depict the approximate expected performance for $ k = 1500 $ Bernoulli trials. In figure (a) the solution fraction $ S^{\alg}_{f} $, in figure (b) the time fraction $ T^{\alg}_{f} $ coefficients are displayed, for the BPP problem (see \textsc{Definition} \ref{Def Solution and Time Performance Coefficients BPP} and $ \alg = \nfd, \ffd, \bfd $. 
		The domain $ N $ is given in logarithmic scale, the values are extracted from the table \ref{Tbl BPP Global Table}.}
	\label{Fig BPP Global figures}
\end{figure}
%
%
\section{The Traveling Salesman Problem TSP}\label{Sec Traveling Salesman Problem TSP}
%
%
The last problem to be analyzed under a proposed Divide-and-Conquer algorithm is the Traveling Salesman Problem TSP. Some graph-theory definitions are recalled in order to describe it.
\begin{definition}\label{Def Graph Theory Theoretical Definitions}
	\begin{enumerate}[(i)]
		\item An edge-weighted, directed graph is a triple $ G = (V, E, d) $, where $ V $ is the set of vertices, $ E \subseteq V \times V $ is the set of (directed) edges or arcs and $ d : E \rightarrow \R^{+} $ is a distance function assigning each arc $ e \in E $ a distance $ d(e) $.
		
		\item A path is a list $ (u_{1},...,u_{k}) $ of vertices $ u_{i} \in V $ for all $ i = 1,..., k $,  holding $ (u_{i}, u_{i + 1}) \in E $ for $ i = 1,..., k - 1 $.
		
		\item A Hamiltonian cycle in $ G $ is a path $ p = (u_{1},...,u_{k}, u_{1}) $ in $ G $ , where $ k = \vert V \vert $ and $ \bigcup^{ k }_{ i = 1 } u_{i} = V $ (i.e., each vertex is visited exactly once except for $ u_{1} $).
		
		\item Given an edge-weighted directed graph $ G = (V, E, d) $ and a vertex subset $ U $ of $ V $, the induced subgraph, denoted by $ G(U) = (U, E^{U}, d^{U}) $ is the subgraph of $ G $ whose vertex set is $ U $, whose edge set consists on all edges in $ G $ that have both endpoints on $ U $ and whose distance function $ d^{U} = d\big\vert_{U \times U}$ is the restriction of the distance function to the set $ U \times U \subseteq V \times V $. 
	\end{enumerate}
\end{definition}
\begin{problem} [The Traveling Salesman Problem, TSP]\label{Pblm TSP Original Problem}
	Given a directed graph $ G = (V, E, d) $ with $ n = \vert V \vert $ vertices, the TSP is to find a Hamiltonian cycle $ ( u_{1}, \ldots, u_{n}, u_{1} ) $ such that the cost $ C(u_{1}, \ldots, u_{n}, u_{1}) = d((u_{k}, u_{1})) + \sum^{ k - 1 }_{ i=1 } d((u_{i}, u_{i + 1})) $ (the total sum of the weights of the edges in the cycle) is minimized.
	\begin{enumerate}[(i)]
		\item The TSP is said to be symmetric if $ d((u, v)) = d((v, u)) $ for all $ u, v \in V $. Otherwise, it is said to be asymmetric.
		
		\item The TSP is said to be metric if $ d((u, v)) \leq d((u, w)) + d((w, v) ) $ for all $ u, v, w \in V $.
	\end{enumerate}
\end{problem}
In the current work it will be assumed that the graph is complete i.e., all possible edges in both directions are present. In addition, three cases are analyzed: 
\begin{enumerate}[a.]
	\item \textbf{MS}: when the TSP is metric and symmetric.
	
	\item \textbf{MA}: when the TSP is metric but not symmetric.
	
	\item \textbf{NMS}: when the TSP is not metric but it is symmetric.
\end{enumerate}

Since it is so widespread and contributes little to our discussion, the formulation of TSP is omitted (see Section 10.3 in \cite{Bertsimas} for a formulation).
%
%
\subsection{A Divide-and-Conquer Approach}\label{Sec Divide and Conquer TSP}
%
In order to apply our Divide-and-Conquer approach on the TSP, it is necessary to introduce an efficiency coefficient to break the original problem in two sub-problems. Hence, in the TSP case, the efficiency for each vertex of the graph, is defined as the sum of the distances of all the edges incident on the node. 
\begin{definition}[TSP efficiency coefficients]\label{Def Efficiency Coefficient TSP}
	Given an edge-weighted complete directed graph $ G = (E, V, d) $, for each node $ u \in V $ its \textbf{efficiency coefficient} is defined as
	\begin{equation}\label{Eq Coefficient TSP}
	g(u) \defining \sum_{v \, \in \, V - \{u\}} d( (u,  v) ) + d( (v, u ) ) . 
	\end{equation}
\end{definition}
Next, a greedy algorithm needs to be introduced in order merge two cycles (the two solution cycles corresponding to each sub-problem: $ \Pi_{\lt}, \Pi_{\rt} $).
\begin{definition}[Cycle Greedy Merging]\label{Def Cycle Greedy Merging}
	Let $ G = (V, E, d) $ be an edge-weighted complete directed with $ \vert V \vert = N $. Let $ V_{\lt}, V_{\rt} $ be a subset partition of $ V$ with $ \vert V_{\lt} \vert = p $ , $ \vert V_{\lt} \vert = q $ and let
	\begin{align}\label{Pblm Cycles for Merging}
	C_{\lt} & = (u_{1},...,u_{p}, u_{1}), &
	C_{\rt} & = (v_{1},...,v_{q}, v_{1}),
	\end{align}
	be two cycles contained in $ V_{\lt} $ and $ V_{\rt} $ respectively. Let $ (u_{i}, u_{i + 1} ) ,  (v_{j}, v_{j + 1} ) $ be the most expensive arcs in $ C_{\lt} $ and $ C_{\rt} $ respectively and break possible ties choosing the lowest index. (Here, the indexes $ i, i + 1 $ and  $ j, j + 1 $ are understood in $ \mod p $ and $ \mod q $ respectively.) Define the merged cycle as
	\begin{equation}\label{Eq Greedy Merged Cycle}
	C = (u_{1}, \ldots, u_{i - 1},\underbrace{ u_{i}, v_{j+1}}, v_{j + 2}, \ldots, 
	v_{q}, v_{1}, \ldots, v_{j - 1}, \underbrace{v_{j}, u_{i + 1}},u_{i + 2} \ldots, u_{p}, u_{1}) .
	\end{equation}
	Where the arcs $ ( u_{i}, v_{j+1} ), ( v_{j}, u_{i + 1} ) $ (marked with underbraces) were included to replace the corresponding arcs $ (u_{i}, u_{i + 1} ) $ and $  (v_{j}, v_{j + 1} ) $.  (See Figure \ref{Fig Merging Cycles Schematics} for an example.)
\end{definition}
\begin{remark}\label{Rem Merging Cycles}
	It must be observed that the merging process introduced in Definition \ref{Def Cycle Greedy Merging} is well-defined, because the choice of arcs to be removed is unique as well as the inclusion of the new arcs, marked with underbraces in \eqref{Eq Greedy Merged Cycle}. Once the arcs $ (u_{i}, u_{i + 1} ) $, $  (v_{j}, v_{j + 1} ) $ are removed, the pair of arcs 
	$ ( u_{i}, v_{j+1} ), ( v_{j}, u_{i + 1} ) $ are the only possible choice for respective replacement that would deliver a joint cycle, given the orientations of $ C_{\lt} $ and $ C_{\rt} $. This observation is particularly important for the asymmetric instances of TSP (MA).
	%
	
	%
	%
\end{remark}
\begin{definition}[A Divide-and-Conquer Algorithm for TSP]\label{Def Divide and Conquer for TSP}
	Let \newline
	$ \Pi = (V, E, d) $ be an instance of \textsc{Problem} \ref{Pblm TSP Original Problem}
	\begin{enumerate}[(i)]
		\item Sort the vertices in increasing order according to their efficiencies, that is
		\begin{equation}\label{Ineq Sorting TSP}
		g(v_{ \pi(1)}) \leq 
		g(v_{ \pi(2)} ) \leq 
		\ldots \leq 
		g(v_{ \pi(N)} ) ,
		\end{equation}
		where $ \pi \in \S_{N} $ is an adequate permutation and $ V = \big\{v_{i} :i \in [N] \big\} = \big\{v_{\pi(i)} :i \in [N] \big\} $.
		
		\item Define $ V_{\lt} =  \{v_{ \pi(i}) \in [N]: i \text{ is odd}\} $ and 
		$ V_{\rt} =  \{ v_{ \pi(i)} \in [N]: i \text{ is even}\} $.  
		
		\item A Divide-and-Conquer pair of \textsc{Problem} \ref{Pblm Original d-KP} is the couple of subproblems $ \big(\Pi_{\side}: \, \side \in \{\lt, \rt\} \big) $, each with input data $ \Pi_{\side} = 
		G(V_{\side}) $, i.e., the induced subgraph introduced in Definition \ref{Def Graph Theory Theoretical Definitions}. In the following, $ \big(\Pi_{\side}, s = \lt, \rt \big) $ is referred as a \textbf{D\&C pair} and denote by $ z_{\side}^{*} $ the optimal solution value of the problem $ \Pi_{\side} $.

		\item The D\&C solution 
		of \textsc{Problem} \ref{Pblm TSP Original Problem} is given by the greedy merging of cycles $ C_{\lt}, C_{\rt} $ (see Definition \ref{Def Cycle Greedy Merging}) furnished by the solution of $ \Pi_{\lt} $ and $ \Pi_{\rt} $ respectively. (See Example \ref{Exm TSP Instance} and Figure \ref{Fig Merging Cycles Schematics} for an illustration.)
	\end{enumerate}
\end{definition}
\begin{example}[Divide-and-Conquer Algorithm on TSP]\label{Exm TSP Instance}
	Consider the TSP instance described by the table \ref{Tbl TSP Example Problem}, with $ N = 6 $ vertices. For this particular case the following order holds: $ g(v_{3}) \leq g(v_{5}) \leq g( v_{2} ) \leq g( v_{6} ) \leq g( v_{1} ) \leq g(v_{4}) $. Therefore, 
	\begin{table}[h!]
		\begin{centering}
			\rowcolors{2}{gray!25}{white}
			\begin{tabular}{c | c c c c c c}
				\rowcolor{gray!80}
				\hline
				\diagbox{$ v_{i} $}{$ v_{j} $} & $ v_{1} $ & $ v_{2} $ & $ v_{3} $ 
				& $ v_{4} $ & $ v_{5} $ & $ v_{6} $	\\
				$ v_{1} $ & 0.00	& 0.61	& 0.10	& 1.08	& 0.46	& 0.11	\\
				$ v_{2} $ & 0.61	& 0.00	& 0.53	& 0.71	& 0.17	& 0.54	\\
				$ v_{3} $ & 0.10	& 0.53	& 0.00	& 0.98	& 0.39	& 0.12	\\
				$ v_{4} $ & 1.08	& 0.71	& 0.98	& 0.00	& 0.83	& 1.07	\\
				$ v_{5} $ & 0.46	& 0.17	& 0.39	& 0.83	& 0.00	& 0.38	\\
				$ v_{6} $ & 0.11	& 0.54	& 0.12	& 1.07	& 0.38	& 0.00	\\
				\hline
				\rowcolor{gray!80}
				$ g(v_{j}) $ & 4.17	& 3.93	& 3.75	& 5.37	& 3.77	& 4.05 \\
				\hline		
			\end{tabular}
			\caption{TSP problem of \textsc{Example} \ref{Exm TSP Instance}, number of items $ N = 6 $. Distance matrix table.}
			\label{Tbl TSP Example Problem}
		\end{centering}
	\end{table}
	\begin{align*}
	& \Pi_{ \lt }: &
	V_{ \lt} & = [ v_{3}, v_{2}, v_{1} ], 
	\\
	& \Pi_{ \rt }: &
	V_{ \rt} & = [v_{5}, v_{6}, v_{4} ] .
	\end{align*}
	Given that this is a symmetric instance (actually MS) of TSP and that $ \Pi_{\lt} $ and $ \Pi_{\rt} $ have both three vertices, there is only one possible solution for each TSP; namely $ C_{\lt} = (v_{1}, v_{2}, v_{3}) $ and $ C_{\rt} = (v_{4}, v_{6}, v_{5}) $ respectively. See Figure \ref{Fig Merging Cycles Schematics}. Next, observe that the most expensive arc in $ C_{\lt} $ is $ (v_{1}, v_{2} ) $, while the heaviest arc in $ C_{\rt} $ is $ ( v_{4}, v_{6} ) $. Hence, these two arcs need to be removed (depicted in dashed line in Figure \ref{Fig Merging Cycles Schematics}) and replaced by $ (v_{1}, v_{6} ) $, $ (v_{4}, v_{2}  ) $ (depicted in blue in Figure \ref{Fig Merging Cycles Schematics}) respectively. Observe that this is the only possible choice in order to preserve the direction of the cycles, in particular, the remaining arcs (depicted in red in Figure \ref{Fig Merging Cycles Schematics}) would fail to build a global Hamiltonian cycle.
	\begin{figure}[h!]
		\centering
		\begin{tikzpicture}
		[scale=1.0,auto=left,every node/.style={}]
		\node (n1) at (-4, 0)  {$ v_{3} $};
		\node (n2) at (-2, 2)  {$ v_{1}  $};
		\node (n3) at (-2, -2) {$ v_{2} $};
		
		\node (n4) at (4, 0)  {$ v_{5} $};
		\node (n5) at (2, 2)  {$ v_{6} $};
		\node (n6) at (2, -2) {$ v_{4} $};
		
		\node (n10) at (1.8, -1.9) {};
		\node (n11) at (1.9, -1.8) {};
		
		\node (n12) at (-1.8, 1.9) {};
		\node (n13) at (-1.9, 1.8) {};
		
		\node (n14) at (-1.8, -1.9) {};
		\node (n15) at (-1.9, -1.8) {};
		
		\node (n16) at (1.9, 1.8)  {};
		\node (n17) at (1.8, 1.9)  {};
		
		\node (n18) at (-2, -2.2) {};
		\node (n19) at (2, -2.2) {};
		
		\node (n20) at (2, 2.2)  { };
		\node (n21) at (-2, 2.2)  { };
		
		\node (n30) at (-1.6, 0)  { 0.61 };
		\node (n31) at (1.6, 0)  { 1.07 };
		
		\node (n32) at (-3.3, 1.2)  { 0.10 };
		\node (n33) at (3.3, 1.2)  { 0.38 };
		
		\node (n34) at (-3.3, -1.2)  { 0.53 };
		\node (n35) at (3.3, -1.2)  { 0.83 };
		
		\node (n40) at (-2.8, 0)  {$ C_{\lt} $};
		\node (n41) at (2.8, 0)  {$ C_{\rt} $};

		\foreach \from/\to in {n1/n2, n3/n1,
			n5/n4, n4/n6}
		\draw[line width = 0.5mm, ->] (\from) -- (\to);
		
		\foreach \from/\to in {n2/n3, n6/n5}
		\draw[line width = 0.5mm, ->, dashed] (\from) -- (\to);
		
		\foreach \from/\to in {n6/n3, n2/n5}
		\draw[color = blue, very thick, ->] (\from) -- (\to);
		
		\foreach \from/\to in {n14/n16, n17/n15, n10/n13, n12/n11,
			n18/n19, n20/n21}
		\draw[color = red , ->] (\from) -- (\to);
		\end{tikzpicture}
		\caption{Merging Cycles Schematics for Example \ref{Exm TSP Instance} (see \ref{Def Cycle Greedy Merging} for a general definition). The cycles $ C_{\lt} = (v_{1}, v_{2}, v_{3}) $ and $ C_{\rt} = (v_{4}, v_{6}, v_{5}) $ are to be merged. The most expensive arcs for each, depicted in dashed line are removed and replaced by the only possible pair, in order to build a global Hamiltonian Cycle. Choosing other pair of arcs (depicted in red) would fail to assemble a cycle due to the orientation.}
		\label{Fig Merging Cycles Schematics} 
	\end{figure}
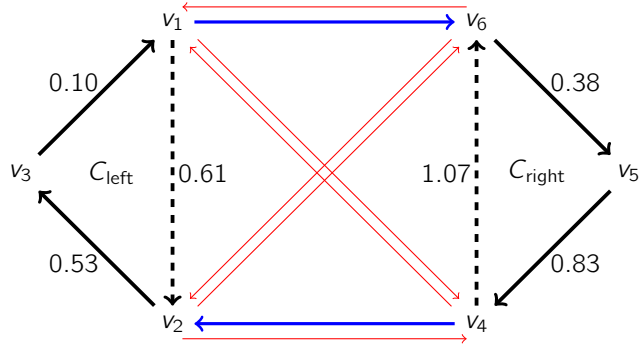
\end{example}
%
%
%
%
\subsection{Numerical Experiments}\label{Sec Numerical Experiments TSP}
%
In order to numerically asses proposed Divide-and-Conquer algorithm's effectiveness on TSP, the numerical experiments are designed according to its main parameters i.e., size ($ N \in \{8, 18,30, 44, 60, 78, 98, 120\} $) and number of trials ($ k $). The number of items was decided based on experience and observation of the method's performance; while the number of trials $ k = 2500 $ was chosen based on \textsc{Equation} \eqref{Eq Bernoulli Trials for Confidence Intervals}. The variance $ \sigma^{2} $ was chosen as the worst possible value for a sample of 150 trials, which delivered an approximate value $ k = 2500 $. (After, executing the 2500 experiments the variance showed consistency with the initial value computed from the 150 trial-sample.)

The random instances were be generated as follows
\begin{enumerate}[a.]
	\item For the MS instances of TSP, $ N $ points in the square $ [0,1] \times [0,1] \subseteq \R^{2} $ are generated (using the uniform distribution) and compute the distance matrix $ D_{\text{MS}} $.
	
	\item For the MA instances of TSP, $ N $ points on the unit circle $ S^{1} = \big\{\vec{x} : \vert \vec{x} \vert = 1 \big\} \subseteq \R^{2} $ are generated (using the uniform distribution) and compute the asymmetric distance matrix $ d_{\text{MA}} $. Given $ \vec{x}, \vec{y} \in S^{1} $, the asymmetric distance $ d_{\text{MA}}(\vec{x}, \vec{y}) $ is defined as the length of the shortest clockwise path from $ \vec{x} $ to $ \vec{y} $ contained in $ S ^{1} $. Clearly $ d_{\text{MA}}(\vec{x}, \vec{y}) \neq d_{\text{MA}}(\vec{y}, \vec{x}) $, unless $ \vec{x} $ and $ \vec{y} $ are antipodal points.
	
	\item For the NMS instances of TSP, an $ N \times N $ upper triangular matrix $ M $ is generated, whose entries are uniformly distributed in $ [0,1] $ and its diagonal is null. The distance matrix is defined as $ D_{\text{NMS}} = M + M^{ T } $.
\end{enumerate}
The performance coefficients are  analogous to those introduced in Definition \ref{Def Solution and Time Performance Coefficients BPP}, given that TSP is a minimization problem as BPP (see also Remark \ref{Ref Solution and Time Performance Coefficients BPP}). 
\begin{definition}[TSP Performance Coefficients]
	\label{Def Solution and Time Performance Coefficients TSP}
	The TSP performance coefficients are given by the percentage fractions of the solution and the computational time given by the proposed Divide-and-Conquer algorithm, with respect to the exact solution, when using the same method of resolution for both cases, i.e., 
	\begin{align}\label{Eq Solution and Time Performance Coefficients TSP}
	& S_{f}^{\Case} = 100 \times \dfrac{z^{\Case}}{z^{\Case}_{\dc}} ,  &
	& T_{f}^{\Case} = 100 \times \dfrac{T_{\dc}^{\Case}}{T^{\Case}}, &
	& \Case = \ms, \, \ma, \, \nms.
	\end{align}
	Here, $ z^{\Case}, z^{\Case}_{\dc} $ indicate the solution furnished for the full problem and for the Divide-and-Conquer algorithm respectively, when solving the corresponding case. 
\end{definition}
The table \ref{Tbl TSP Global Table} below summarizes the expected values of the performance coefficients $ S^{\Case}_{f} $ (solution fraction) $ T^{\Case}_{f} $ (time fraction). The corresponding graphs are depicted in the figures \ref{Fig TSP Global figures} (a) for $ S^{\Case}_{f} $ and (b) for $ T^{\Case}_{f} $, respectively.
\begin{table}[h]
	\begin{centering}
		\rowcolors{2}{gray!25}{white}
		\begin{tabular}{c | cc  cc  cc}
			\rowcolor{gray!80}
			\hline
			Items & \multicolumn{2}{ c  }{MS} & \multicolumn{2}{ c }{MA}
			& \multicolumn{2}{ c }{NMS} \\
			\rowcolor{gray!80} $ N $ 
			& $  S^{\ms}_{f} $ & $ T^{\ms}_{f} $ 
			& $  S^{\ma}_{f} $ & $ T^{\ma}_{f} $ 
			& $  S^{\nms}_{f} $ & $ T^{\nms}_{f} $ 
			\\
			8	& 70.15	& 75.66	& 64 & 106.98 & 59.62 & 86.32 \\
			18	& 65.72	& 79.64	& 61.37	& 69.37	& 60.59	& 105.08 \\
			30	& 68.2	& 64.61	& 87.9	& 53.11	& 67.05	& 92.1 \\
			44	& 70.26	& 67.7	& 84.12	& 95.54	& 70.74	& 80.52 \\
			60	& 70.04	& 62.68	& 98.85	& 84.4	& 74.82	& 75.35 \\
			78	& 69.69	& 55.43	& 94.45	& 66.68	& 77.22	& 63.11 \\
			98	& 70.24	& 56.6	& 97.86	& 68.91	& 78.65 & 70.58 \\
			120	& 71.52	& 50.57 & 99.94	& 77.16 & 78.44	& 57.31 \\
			\hline		
		\end{tabular}
		\caption{Performance table for the Divide-and-Conquer method, solving the TSP problem for the three types of instances: MS, MA and NMS. Bernoulli trials $ k = 2500 $. The performance coefficients $ S^{\case}_{f} $ and  $ T^{\case}_{f} $ for $ \case =  \ms, \ma, \nms $ are introduced in \textsc{Equation} \eqref{Eq Solution and Time Performance Coefficients TSP}.}
		\label{Tbl TSP Global Table}
	\end{centering}
\end{table}
\begin{figure}[h!]
	\begin{subfigure}[Solution Fraction $ S_{f} $ Performance.]
		{\includegraphics[scale = 0.5]{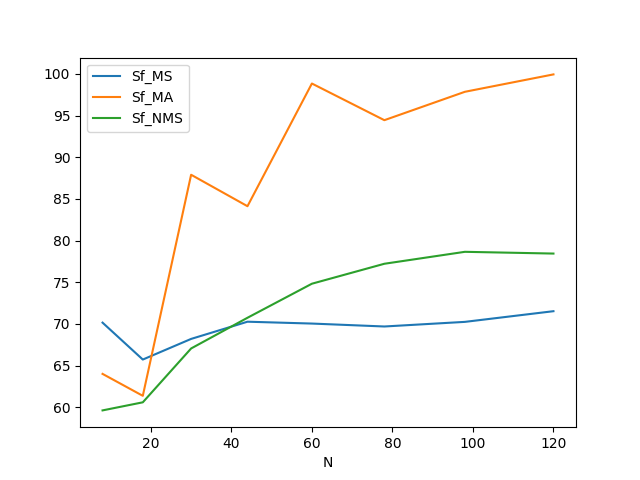}}
	\end{subfigure}
	\begin{subfigure}[Time Fraction $ T_{f} $ Performance.]
		{\includegraphics[scale = 0.5]{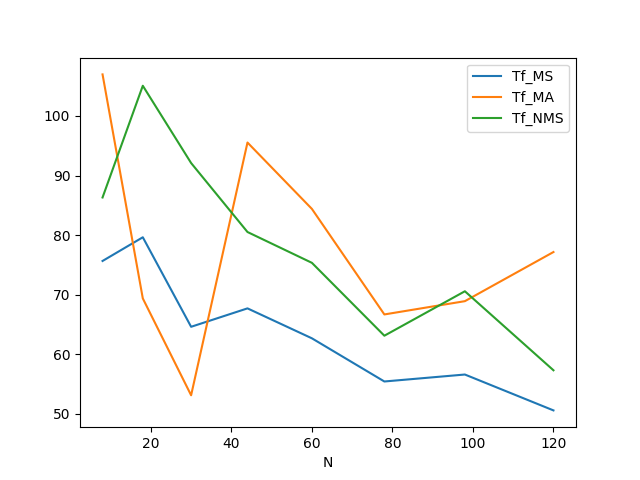}}
	\end{subfigure}
	\caption{The graphs above depict the approximate expected performance for $ k = 2500 $ Bernoulli trials. In figure (a) the solution fraction $ S^{\case}_{f} $, in figure (b) for the time fraction $ T^{\case}_{f} $ coefficients are displayed, for the three types of instance of TSP $ \case = \ms, \ma, \nms $. 
		Unlike the previous examples, the domain $ N $ is given in linear scale. The values are extracted from the table \ref{Tbl TSP Global Table}.}
	\label{Fig TSP Global figures}
\end{figure}
%
%
%
%
%
\section{Conclusions and Closing Discussion}\label{Sec Conclusions and Final Discussion}
%
%
The present work yields several conclusions which are listed below
\begin{enumerate}[(i)]
	\item  For d-KP, the tables \ref{Tbl d-KP, t = 0.25}, \ref{Tbl d-KP, t = 0.5}, \ref{Tbl d-KP, t = 0.75}, as well as their corresponding figures \ref{Fig DC d-KP, t = 0.25}, \ref{Fig DC d-KP, t = 0.5}, \ref{Fig DC d-KP, t = 0.75} show that the proposed algorithm is recommendable, starting from certain number of items $ N_{0} $. The latter, according to the required precision for the particular problem at hand. It is also clear that the tightness plays a very important role in defining the number $ N_{0} $, starting from which the algorithm is useful. Although the dimension $ D $ also has an impact on $ N_{0} $, it is clear that for different values of $ D $, the expected values of the method's performance get close to each other for $ N \geq 100 $, which is reasonably low.
	
	\item In the case of the BPP, it is clear that the proposed method has a poor interaction with the $ \nfd $ algorithm, because the required computational time is always above 100\% and a precision price is always paid. On the other hand, the $ \ffd $ algorithm complements very well the Divide-and-Conquer algorithm, giving substantial reduction of computational time for a low number of items, namely $ N \geq 50 $. Finally, the $ \bfd $ algorithm is a bad complement for the analyzed strategy, when compared with $ \ffd $ algorithm. The latter, considering that $ \bfd $ yields a marginal improvement of precision with respect to $ \ffd $, while using a significant higher fraction of computational time (although not as much as $ \nfd $). It is to be noted that the three algorithms yield significantly high precision values, even for low values of $ N $ and that ranking their usefulness is decided based on its computational time. All of the previous can be observed in the table \ref{Tbl BPP Global Table} and the figure \ref{Fig BPP Global figures}. 
	
	\item Analyzing the results for TSP, reported in table \ref{Tbl TSP Global Table} and figure \ref{Fig TSP Global figures}, it is immediate to conclude that the proposed method performs poorly for the $ \ma $ case. It shows an erratic curve of computational time, although it attains a good precision level for $ N \geq 50 $ vertices. As for the other two cases, although they show reasonable computational time fraction for $ N \geq 80 $, the precision fraction never goes above 80\%, which is very low value to consider the proposed algorithm as a valid option for solving TSP. Therefore, it follows that the method is not recommendable for TSP. 
	
	\item For the TSP, it is direct to see that the greedy function introduced, does not capture the structure/essence of the problem. An alternative way would be to define the subsets 
	$ V_{\lt}, V_{\rt} $ using a distance criterion for the metric TSP (MS). Namely, locate a pair of vertices $ u, v \in V $ satisfying  $ \Vert u - v \Vert =  \text{diameter}(V) $; then define $ V_{\lt} $ as the half of vertices closer to $ u $ (breaking ties with the index) and $ V_{\rt} = V - V_{\lt} $. However, the following must be observed. First, this approach can only be used in MS. Second, when implemented its results are no better than those presented in \cite{TSPApplegatChained, TSPFritzkeDC, TSPKazuonoriDC} and \cite{TSPJohnson2007}.  
	
	\item It must be also be reported that, without parallel implementation, the proposed algorithms not recommendable for d-KP and/or BPP, due to the involved computational time.   	
	
	\item Although the proposed Divide-and-Conquer strategy has proved to be useful for d-KP, its extension to other generalized Knapsack Problems may not be direct. It must be stressed that the tightness of the problem (introduced in Definition \ref{Def Tightness}) is strongly correlated with the performance of the method for d-KP. Analogously, other generalized Knapsack Problems have to be studied, as it was done here and observe the impact of those parameters containing structural information of the problem. 
	
	\item Deeper studies from the analytical point of view, analogous to those presented in \cite{MoralesMartinez2023}, will be pursued in the future for d-KP and BPP. In order to theoretically establish the expected behavior as well as furnishing a quality certificate for the algorithms presented here, adequate probabilistic models and spaces will be introduced and analyzed. 
	
\end{enumerate}
%
%
%
%
\section*{Acknowledgements}
\noindent The Author wishes to thank Universidad Nacional de Colombia, Sede Medell\'in for supporting the production of this work through the project Hermes 54748 as well as granting access to Gauss Server, financed by ``Proyecto Plan 150x150 Fomento de la cultura de evaluaci\'on continua a trav\'es del apoyo a planes de mejoramiento de los programas curriculares" (\url{gauss.medellin.unal.edu.co}), where the numerical experiments were executed. 


\label{up}


\end{document}